\documentclass[12pt,fleqn]{article}

\usepackage{amsmath}
\usepackage{amsthm}
\usepackage{amssymb}
\usepackage{amsfonts}
\usepackage{epsf}
\usepackage{graphicx}
\usepackage{epstopdf}
\usepackage{color}
\usepackage{mathptmx}

\newtheorem{theorem}{Theorem}[section]
\newtheorem{lemma}[theorem]{Lemma}

\theoremstyle{definition}

\newtheoremstyle{named}{}{}{\itshape}{}{\bfseries}{.}{.5em}{\thmnote{#3's }#1} \theoremstyle{named} 

\theoremstyle{remark}
\newtheorem{remark}[theorem]{Remark}

\numberwithin{equation}{section}

\numberwithin{equation}{section}

\title{Existence of a lower bound for the distance between point masses of relative equilibria in spaces of constant curvature}
\author{Pieter Tibboel$^\ast$}
\begin{document}
\maketitle
\begin{abstract}
  We prove that if for the curved $n$-body problem the masses are given, the minimum distance between the point masses of a specific type of relative equilibrium solution to that problem has a universal lower bound that is not equal to zero. We furthermore prove that the set of all such relative equilibria is compact. This class of relative equilibria includes all relative equilibria of the curved $n$-body problem in $\mathbb{S}^{2}$, $\mathbb{H}^{2}$ and a significant subset of the relative equilibria for $\mathbb{S}^{3}$ and $\mathbb{H}^{3}$.
\end{abstract}
\begin{description}

\item \hspace*{3.8mm}$\ast$ Department of Mathematics, City University of
Hong Kong, Hong Kong. \\
Email: \texttt{ptibboel@cityu.edu.hk}

\end{description}

\newpage

\section{Introduction. }
  By $n$-body problems, we mean problems where we want to find the dynamics of $n$ point particles. By relative equilibria, we mean solutions to such problems where the point particles represent rotating configurations of fixed size and shape. The \textit{$n$-body problem in spaces of constant curvature}, or \textit{curved $n$-body problem} is an extension of the Newtonian $n$-body problem (in Euclidean space) into spaces of nonzero, constant Gaussian curvature, which means that the space is either spherical (if the curvature is positive), or hyperbolical (if the curvature is negative) (see \cite{DPS1}, \cite{DPS2} and \cite{DPS3}). It was noted in \cite{D4} and \cite{DK} that it suffices to consider the case that the curvature is equal to either $+1$, or $-1$. More precisely, following \cite{DPS1}, \cite{DPS2}, \cite{DPS3}, \cite{D3} and \cite{DK}, if we define the space
  \begin{align*}
    \mathbb{M}_{\sigma}^{k}=\{(x_{1},....,x_{k+1})\in\mathbb{R}^{k+1}|x_{1}^{2}+...+x_{k}^{2}+\sigma x_{k+1}^{2}=\sigma\},
  \end{align*}
  where $\sigma$ equals either $+1$, or $-1$
  and for $x$, $y\in\mathbb{M}_{\sigma}^{k}$ define the inner product
  \begin{align*}
    x\odot_{k} y=x_{1}y_{1}+...+x_{k}y_{k}+\sigma x_{k+1}y_{k+1},
  \end{align*}
we mean the problem of finding the dynamics of $n$ point particles with respective masses $m_{1}$,..., $m_{n}$ and coordinates $q_{1}$,..., $q_{n}\in\mathbb{M}_{\sigma}^{k}$, $k\geq 2$, as described by the system of differential equations
\begin{align}\label{EquationsOfMotion}
   \ddot{q}_{i}=\sum\limits_{j=1,\textrm{ }j\neq i}^{n}\frac{m_{j}(q_{j}-\sigma(q_{i}\odot_{k} q_{j})q_{i})}{(\sigma-\sigma(q_{i}\odot_{k} q_{j})^{2})^{\frac{3}{2}}}-\sigma(\dot{q}_{i}\odot_{k}\dot{q}_{i})q_{i},\textrm{ }i\in\{1,...,n\}.
\end{align}
The first to investigate $n$-body problems for spaces of constant curvature were Bolyai \cite{BB} and Lobachevsky \cite{Lo}, who independently proposed a curved 2-body problem in hyperbolic space $\mathbb{H}^{3}$ in the 1830s. Since then, $n$-body problems for spaces of constant curvature have been studied by mathematicians such as Dirichlet, Schering \cite{S1}, \cite{S2}, Killing \cite{K1}, \cite{K2}, \cite{K3}, Liebmann \cite{L1}, \cite{L2}, \cite{L3} and more recently Kozlov and Harin \cite{KH}. However, the study of $n$-body problems in spaces of constant curvature for the case that $n\geq 2$ started with \cite{DPS1}, \cite{DPS2}, \cite{DPS3}  by Diacu, P\'erez-Chavela and Santoprete. After this breakthrough, additional results for the $n\geq 2$ case were then obtained by Cari\~nena, Ra\~nada, Santander \cite{CRS}, Diacu \cite{D1}, \cite{D2}, \cite{D3}, Diacu, Kordlou \cite{DK}, Diacu, P\'erez-Chavela \cite{DP}. For a more detailed historical overview, please see \cite{D2}, \cite{D3}, \cite{D4}, \cite{DK}, or \cite{DPS1}.

M. Shub proved for the Newtonian $n$-body problem that if we fix the masses and the angular velocity (i.e. the speed with which the angle of the rotation changes), the set of possible relative equilibria is compact and as a direct consequence that there exists a universal nonzero lower bound for the distance between the point particles of the relative equilibria in such a set (see \cite{Shub}). Shub's results were a potential first step in what may lead to a proof of the famous sixth Smale problem (see \cite{Smale}) which states that such sets are not only compact, but, in fact, finite.

  In this paper, following Shub's line of thought, we will make a first attempt at investigating to which extent we can extend his results to the constant curvature case. More specifically, for
\begin{align*}
  T(t)=\begin{pmatrix}
    \cos(t) & -\sin(t) \\
    \sin(t) & \cos(t)
  \end{pmatrix}
\end{align*}
 a $2\times 2$ rotation matrix, $A>0$, $Q_{1}$,..., $Q_{n}\in\mathbb{R}^{2}$ and $Z\in\mathbb{R}^{k-1}$ constant, if we call any solution $q_{1}$,..., $q_{n}$ of (\ref{EquationsOfMotion}) of the form
\begin{align}\label{Expression relative equilibrium}
  q_{i}(t)=\begin{pmatrix}T(At)Q_{i}\\ Z\end{pmatrix}
\end{align}
a \textit{relative equilibrium} and $A$ its \textit{angular velocity}, then we will prove that if  $\|\cdot\|_{k}$ is the Euclidean norm on $\mathbb{R}^{k}$ that
\begin{theorem}\label{Main Theorem}
  There exists a universal constant $C>0$ such that for any relative equilibrium solution  of (\ref{EquationsOfMotion}) $\|q_{i}-q_{j}\|_{k}>C$ for all \\$i$, $j\in\{1,...,n\}$, $i\neq j$ if the masses $m_{1}$,..., $m_{n}$ are given.
\end{theorem}
and
\begin{theorem}\label{Main Corollary}
  If we write any set of vectors $Q_{1}$,..., $Q_{n}\in\mathbb{R}^{2}$ of a relative equilibrium solution $q_{1}$,..., $q_{n}$ as a $2n$-dimensional vector
  \begin{align*}
     \begin{pmatrix}
       Q_{1}\\
       \vdots\\
       Q_{n}
     \end{pmatrix},
   \end{align*}
   then the set of all such $2n$-dimensional vectors, for fixed masses $m_{1}$,..., $m_{n}$ and angular velocity $A$, is compact in $\mathbb{R}^{2n}$.
\end{theorem}
\begin{remark}
  Note that the definition of a relative equilibrium used in Theorem~\ref{Main Theorem} and Theorem~\ref{Main Corollary} includes all relative equilibria of the $n$-body problem in $\mathbb{M}_{\sigma}^{2}$ and a subclass of the positive elliptic relative equilibria in $\mathbb{S}^{3}$ as defined in \cite{D3} and a subclass of the negative elliptic relative equilibria in $\mathbb{H}^{3}$ as defined in \cite{D3}, which are two out of all four possible classes of relative equilibria in $\mathbb{M}_{\sigma}^{3}$ (see \cite{D3}).
\end{remark}

 We will first formulate two lemmas, which will be done in section~\ref{section lemma}, that are  related to Criterion~1 in \cite{D2} and then use those lemmas to prove Theorem~\ref{Main Theorem} in section~\ref{Proof of Main Theorem} and Theorem~\ref{Main Corollary} in section~\ref{Proof of Main Corollary}.
\section{Background theory}\label{section lemma}
In order to formulate the aforementioned lemmas we need for the proofs of Theorem~\ref{Main Theorem} and Theorem~\ref{Main Corollary}, we need to introduce some notation:\\
Let $m\in\mathbb{N}$. Let $\langle\cdot,\cdot\rangle_{m}$ be the Euclidean inner product on $\mathbb{R}^{m}$.
Let $i$, $j\in\{1,...,n\}$. Let
\begin{align*}
    q_{1}(t)=\begin{pmatrix}
      T(At)Q_{1}\\Z
    \end{pmatrix},...,\textrm{ }q_{n}(t)=\begin{pmatrix}
      T(At)Q_{n}\\Z
    \end{pmatrix}
  \end{align*}
  be a relative equilibrium, define $r:=\|Q_{i}\|$ for all $i\in\{1,...,n\}$ and let $\alpha_{i}$ be the angle between $Q_{i}$ and the first coordinate axis.
Then the first lemma we will need is:
\begin{lemma}
  Let $q_{1}$,..., $q_{n}$ be a relative equilibrium solution as in (\ref{Expression relative equilibrium}). Then
  \begin{align}\label{c}
    0=\sum\limits_{j=1, j\neq i}^{n}\frac{m_{j}\sin{(\alpha_{i}-\alpha_{j})}}{(1-\cos{(\alpha_{i}-\alpha_{j})})^{\frac{3}2}(2-\sigma r^{2}(1-\cos{(\alpha_{i}-\alpha_{j})}))^{\frac{3}{2}}}
  \end{align}
\end{lemma}
\begin{proof}
  This lemma is a direct consequence of Criterion~1 in \cite{T2}, but the proof for our case is very short, which is why it has been added here regardless:

  Inserting our expressions for $q_{1}$,..., $q_{n}$ into (\ref{EquationsOfMotion}), using that $(T(At))''=-A^{2}T(At)$ and that $(T(At))'=AT(At)\begin{pmatrix}
    0 & -1\\ 1 & 0
  \end{pmatrix}$, gives
  \begin{align}\label{To prove the lemma Identity 1}
    \begin{pmatrix}
      -A^{2}T(At)Q_{i}\\
      \overrightarrow{0}
    \end{pmatrix}&=\sum\limits_{j=1,\textrm{ }j\neq i}^{n}\frac{m_{j}\left(\begin{pmatrix}
      T(At)Q_{j}\\ Z
    \end{pmatrix}-\sigma(q_{i}\odot_{k}q_{j})\begin{pmatrix}
      T(At)Q_{i}\\ Z
    \end{pmatrix}\right)}{(\sigma-\sigma(q_{i}\odot_{k}q_{j})^{2})^{\frac{3}{2}}}\nonumber\\
    &-\sigma(\dot{q}_{i}\odot_{k}\dot{q}_{i})\begin{pmatrix}
      T(At)Q_{i}\\ Z
    \end{pmatrix},\textrm{ }i\in\{1,...,n\},
  \end{align}
  where $\overrightarrow{0}\in\mathbb{R}^{k-2}$. Writing out the identities for the first two coordinates of the vectors of (\ref{To prove the lemma Identity 1}) gives
  \begin{align}\label{To prove the lemma Identity 2}
    -A^{2}T(At)Q_{i}&=\sum\limits_{j=1,\textrm{ }j\neq i}^{n}\frac{m_{j}\left(T(At)Q_{j}-\sigma(q_{i}\odot_{k} q_{j})T(At)Q_{i}\right)}{(\sigma-\sigma(q_{i}\odot_{k}q_{j})^{2})^{\frac{3}{2}}}\nonumber\\
    &-\sigma(\dot{q}_{i}\odot_{k}\dot{q}_{i})T(At)Q_{i},\textrm{ }i\in\{1,...,n\}.
  \end{align}
  Multiplying both sides of (\ref{To prove the lemma Identity 2}) with $(T(At))^{-1}$ and consequently taking inner products at both sides with
  \begin{align*}
    \begin{pmatrix}
      0 & -1 \\
      1 & 0
    \end{pmatrix}Q_{i}
  \end{align*}
  gives
  \begin{align*}
    0=\sum\limits_{j=1,\textrm{ }j\neq i}^{n}\frac{m_{j}\langle Q_{j},\begin{pmatrix}
      0 & -1 \\
      1 & 0
    \end{pmatrix}Q_{i}\rangle_{2}}{(\sigma-\sigma(q_{i}\odot_{k}q_{j})^{2})^{\frac{3}{2}}},\textrm{ }i\in\{1,...,n\},
  \end{align*}
  which can be rewritten as
  \begin{align}\label{To prove the lemma Identity 3}
    0=\sum\limits_{j=1,\textrm{ }j\neq i}^{n}\frac{m_{j}\|Q_{j}\|\|Q_{i}\|\sin{(\alpha_{j}-\alpha_{i})}}{(\sigma-\sigma(\|Q_{j}\|\|Q_{i}\|\cos{(\alpha_{j}-\alpha_{i})}+Z\odot_{k-2}Z)^{2})^{\frac{3}{2}}},\textrm{ }i\in\{1,...,n\}.
  \end{align}
  Using that $\sigma=q_{i}\odot_{k} q_{i}=\|Q_{i}\|_{2}^{2}+Z\odot_{k-2}Z$ and that $\|Q_{i}\|=\|Q_{j}\|=r$ allows us to rewrite (\ref{To prove the lemma Identity 3}) as
  \begin{align*}
    0=\sum\limits_{j=1,\textrm{ }j\neq i}^{n}\frac{m_{j}r^{2}\sin{(\alpha_{j}-\alpha_{i})}}{(\sigma-\sigma(r^{2}\cos{(\alpha_{j}-\alpha_{i})}+\sigma-r^{2})^{2})^{\frac{3}{2}}},\textrm{ }i\in\{1,...,n\},
  \end{align*}
  which means that
  \begin{align*}
    0=\sum\limits_{j=1,\textrm{ }j\neq i}^{n}\frac{m_{j}\sin{(\alpha_{j}-\alpha_{i})}}{(1-\cos{(\alpha_{j}-\alpha_{i})})^{\frac{3}{2}}(2-\sigma r^{2}(1-\cos{(\alpha_{j}-\alpha_{i})})^{2})^{\frac{3}{2}}},\textrm{ }i\in\{1,...,n\},
  \end{align*}
  which completes the proof.
\end{proof}
\begin{lemma}\label{Lemma for corollary}
  For any relative equilibrium solution to (\ref{EquationsOfMotion})
  \begin{align*}\begin{pmatrix}T(At)Q_{i}\\Z\end{pmatrix},\textrm{ }i\in\{1,...,n\}\end{align*}
  we have that if $Z\neq 0$, then
  \begin{align*}
    \sigma A^{2}r^{2}=\sum\limits_{j=1,\textrm{ }j\neq i}^{n}\frac{m_{j}\left(1-\sigma(q_{i}\odot_{k}q_{j})\right)}{(\sigma-\sigma(q_{i}\odot_{k}q_{j})^{2})^{\frac{3}{2}}},\textrm{ }i\in\{1,...,n\}.
  \end{align*}
\end{lemma}
\begin{proof}
  Because of (\ref{To prove the lemma Identity 1}), we have that
  \begin{align*}
      \overrightarrow{0}=\sum\limits_{j=1,\textrm{ }j\neq i}^{n}\frac{m_{j}\left(Z-\sigma(q_{i}\odot_{k}q_{j})Z\right)}{(\sigma-\sigma(q_{i}\odot_{k}q_{j})^{2})^{\frac{3}{2}}}-\sigma(\dot{q}_{i}\odot_{k}\dot{q}_{i})Z,\textrm{ }i\in\{1,...,n\},
  \end{align*}
  which can be rewritten as
  \begin{align}\label{To prove the second lemma Identity 1}
      \sigma(\dot{q}_{i}\odot_{k}\dot{q}_{i})Z=\sum\limits_{j=1,\textrm{ }j\neq i}^{n}\frac{m_{j}\left(Z-\sigma(q_{i}\odot_{k}q_{j})Z\right)}{(\sigma-\sigma(q_{i}\odot_{k}q_{j})^{2})^{\frac{3}{2}}},\textrm{ }i\in\{1,...,n\}.
  \end{align}
  Because $Z\neq 0$, there has to be at least one nonzero entry of $Z$, so if we divide the identity in (\ref{To prove the second lemma Identity 1}) for that entry by that entry, we get
  \begin{align}\label{To prove the second lemma Identity 2}
      \sigma(\dot{q}_{i}\odot_{k}\dot{q}_{i})=\sum\limits_{j=1,\textrm{ }j\neq i}^{n}\frac{m_{j}\left(1-\sigma(q_{i}\odot_{k}q_{j})\right)}{(\sigma-\sigma(q_{i}\odot_{k}q_{j})^{2})^{\frac{3}{2}}},\textrm{ }i\in\{1,...,n\}.
  \end{align}
  Because $\dot{q}_{i}\odot_{k}\dot{q}_{i}=A^{2}r^{2}$, this proves the lemma.
\end{proof}
\section{Proof of Theorem~\ref{Main Theorem}}\label{Proof of Main Theorem}
\begin{proof}
  Assume that the contrary is true. Then there exist sequences

  $\{Q_{ip}\}_{p=1}^{\infty}\subset\mathbb{R}^{2}$, $i=1,...,n$, with respective sequences of relative equilibria \begin{align*}\left\{\begin{pmatrix}T(A_{p}t)Q_{ip}\\Z_{p}\end{pmatrix}\right\}_{p=1}^{\infty},\textrm{ }i\in\{1,...,n\}\end{align*}
  for which, after renumbering the
  \begin{align*}
    \begin{pmatrix}T(A_{p}t)Q_{ip}\\Z_{p}\end{pmatrix}
  \end{align*}
  in terms of $i$ if necessary, there exists an $l\in\{1,...,n\}$, such that
  \begin{align*}
    \begin{pmatrix}T(A_{p}t)Q_{1p}\\Z_{p}\end{pmatrix},...,\textrm{ }\begin{pmatrix}T(A_{p}t)Q_{lp}\\Z_{p}\end{pmatrix}
  \end{align*}
  go to the same limit for $p$ going to infinity. \\
  For each of those $p$, we have because of Lemma~\ref{c} that
  \begin{align}\label{partytime}
    0=\sum\limits_{j=1, j\neq i}^{n}\frac{m_{j}\sin{(\alpha_{ip}-\alpha_{jp})}}{(1-\cos{(\alpha_{ip}-\alpha_{jp})})^{\frac{3}2}(2-\sigma r_{p}^{2}(1-\cos{(\alpha_{ip}-\alpha_{jp})}))^{\frac{3}{2}}},
  \end{align}
  where $\alpha_{ip}$ and $\alpha_{jp}$ are the angles between the first coordinate axis and $Q_{ip}$ and the angle between the first coordinate axis and $Q_{jp}$ respectively and $r_{p}=\|Q_{ip}\|$. \\
  Because of (\ref{partytime}), we thus get that
  \begin{align}\label{partytime2}
    0&=\sum\limits_{j=2}^{l}\frac{m_{j}\sin{(\alpha_{1p}-\alpha_{jp})}}{(1-\cos{(\alpha_{1p}-\alpha_{jp})})^{\frac{3}2}(2-\sigma r_{p}^{2}(1-\cos{(\alpha_{1p}-\alpha_{jp})}))^{\frac{3}{2}}}\nonumber\\
    &+\sum\limits_{j=l+1}^{n}\frac{m_{j}\sin{(\alpha_{1p}-\alpha_{jp})}}{(1-\cos{(\alpha_{1p}-\alpha_{jp})})^{\frac{3}2}(2-\sigma r_{p}^{2}(1-\cos{(\alpha_{1p}-\alpha_{jp})}))^{\frac{3}{2}}}.
  \end{align}
  There are two possibilities:
  \begin{itemize}
    \item[1. ] $\alpha_{1p}-\alpha_{jp}$ goes to zero for $j\in\{1,...,l\}$ and $r_{p}$ is bounded for $p$ going to infinity.
    \item[2. ] $\alpha_{1p}-\alpha_{jp}$ goes to zero for $j\in\{1,...,l\}$ and $r_{p}$ is not bounded for $p$ going to infinity.
  \end{itemize}
  For the first case, note that by l'H\^opital and by renumbering the $\alpha_{ip}$ in terms of $i$ and taking subsequences if necessary such that $\alpha_{1p}-\alpha_{jp}$ decreases to zero for all $j\in\{1,...,l\}$ that
  \begin{align}\label{lhopital}
    \lim\limits_{(\alpha_{1p}-\alpha_{jp})\downarrow 0}\frac{m_{j}\sin{(\alpha_{1p}-\alpha_{jp})}}{1-\cos{(\alpha_{1p}-\alpha_{jp})}}=\lim\limits_{(\alpha_{1p}-\alpha_{jp})\downarrow 0}\frac{m_{j}\cos{(\alpha_{1p}-\alpha_{jp})}}{\sin{(\alpha_{1p}-\alpha_{jp})}}=+\infty,
  \end{align}
  which means that if we take the limit where $p$ goes to infinity on both sides of (\ref{partytime2}), we get that $0=\infty$, which is a contradiction.

  For the second case, the $n$-body problem is defined on $\mathbb{H}^{k}$ and thus $\sigma=-1$. Then multiplying both sides of (\ref{partytime2}) with $r_{p}^{3}$ and noting that for $p$ going to infinity
  \begin{align*}
    \frac{r_{p}^{3}}{(2-\sigma r_{p}^{2}(1-\cos{(\alpha_{1p}-\alpha_{jp})}))^{\frac{3}{2}}}=\frac{r_{p}^{3}}{(2+r_{p}^{2}(1-\cos{(\alpha_{1p}-\alpha_{jp})}))^{\frac{3}{2}}}
  \end{align*}
  does not go to zero, leads, combined with (\ref{lhopital}), to the desired contradiction we got for the first case. This completes the proof.
\end{proof}
\section{Proof of Theorem~\ref{Main Corollary}}\label{Proof of Main Corollary}
Assume that the contrary is true. Then there exist sequences

$\{Q_{ip}\}_{p=1}^{\infty}$, $i\in\{1,...,n\}$ and corresponding relative equilibria
\begin{align*}
q_{ip}=\begin{pmatrix}T(A)Q_{ip}\\Z_{p}\end{pmatrix},\textrm{ }i\in\{1,...,n\}\end{align*}
where $q_{1p}$,..., $q_{np}$ solve (\ref{EquationsOfMotion}), such that $r_{p}:=\|Q_{ip}\|$ goes to infinity for $p$ going to infinity. \\
As consequently, for $p$ large enough, taking subsequences if necessary, $Z_{p}\neq 0$, we have by Lemma~\ref{Lemma for corollary} that
\begin{align}\label{compactness trick}
  \sigma A^{2}r_{p}^{2}=\sum\limits_{j=1,\textrm{ }j\neq i}^{n}\frac{m_{j}\left(1-\sigma(q_{ip}\odot_{k}q_{jp})\right)}{(\sigma-\sigma(q_{ip}\odot_{k}q_{jp})^{2})^{\frac{3}{2}}} ,\textrm{ }i\in\{1,...,n\}.
\end{align}
Letting $p$ go to infinity on both sides of (\ref{compactness trick}) means that the left-hand  side of (\ref{compactness trick}) goes to infinity, which is only possible if the right-hand side of (\ref{compactness trick}) does the same. The right-hand side of (\ref{compactness trick}) can only become infinitely large if for at least one term
\begin{align*}
  \frac{m_{j}\left(1-\sigma(q_{ip}\odot_{k}q_{jp})\right)}{(\sigma-\sigma(q_{ip}\odot_{k}q_{jp})^{2})^{\frac{3}{2}}}
\end{align*}
the denominator goes to zero, which means that $\lim\limits_{p\rightarrow\infty}q_{ip}\odot_{k}q_{jp}=-1$, which means that $q_{ip}$ and $q_{jp}$ have the same limit. This contradicts Theorem~\ref{Main Theorem}.
%\section{Acknowledgements}
%The author is indebted to Florin Diacu and Dan Dai for all their advice.

\end{document}